\documentclass[12pt]{article}

\usepackage{amssymb,latexsym,amsmath,mathrsfs}
\usepackage{amsfonts}
\usepackage{amscd}
\usepackage{dsfont}
\usepackage{bbm}
\usepackage{rotating}
\usepackage{newlfont}
\usepackage{enumitem}

\usepackage[numbers,sort&compress]{natbib}

\date{}

\makeatother

\def\h25{\hspace{-.3cm}}

\def\leq{\leqslant}
\def\geq{\geqslant}

\newtheorem{thm}{Theorem}[section]

\newtheorem{lem}{Lemma}[section]

\newtheorem{obs}{Observation}[section]

\newenvironment{proof}{\noindent {\bf
Proof.}}{\hfill\rule{3mm}{3mm}\par\medskip}

\begin{document}

\title{
Maximal density and the kappa values for the families $\{a,a+1,2a+1,n\}$ and $\{a,a+1,2a+1,3a+1,n\}$
}
\author
{Ram Krishna Pandey
\thanks{Indian Institute of Technology Roorkee--247667, India. Email: ram.pandey@ma.iitr.ac.in}
\and
Neha Rai
\thanks{Indian Institute of Technology Roorkee--247667, India. Email: neha.rai0184@gmail.com, nrai@ma.iitr.ac.in}
}
\date{\today}
\maketitle

\begin{abstract}
Let $M$ be a set of positive integers. We study the maximal density $\mu(M)$ of the sets of nonnegative integers $S$ whose elements do not differ by an element in $M$. In 1973, Cantor and Gordon established a formula for $\mu(M)$ for $|M|\leq 2$. Since then, many researchers have worked upon the problem and found several partial results in the case $|M|\geq 3$, including some results in the case, $M$ is an infinite set. In this paper, we study the maximal density problem for the families $M=\{a,a+1,2a+1,n\}$ and $M=\{a,a+1,2a+1,3a+1,n\}$, where $a$ and $n$ are positive integers. In most of the cases, we find bounds for the parameter \textit{kappa}, denoted by $\kappa(M)$, which actually serves as a lower bound for $\mu(M)$. The parameter $\kappa(M)$ has already got its importance due to its rich connection with the problems such as the ``lonely runner conjecture" in Diophantine approximations and coloring parameters such as ``circular coloring" and ``fractional coloring" in graph theory.
\end{abstract}

\bigskip

\section{Introduction}
%
For a given set $M$ of positive integers, a problem of Motzkin asks to find the maximal upper density of sets $S$ of non-negative integers in which no two elements of $S$ are allowed to differ by an element of $M.$ Following the question of Motzkin, if $M$ is a given set of positive integers, a set $S$ of non-negative integers is said to be an {\it $M$-set} if $ a, b\in S$, then $a- b \notin M .$ For $x\in \mathbb{R}$ and a set $S$ of non-negative integers, let $S(x)$ be the number of elements $n\in S$ such that $n\leq x$. We define the upper and lower densities of $S$, denoted respectively by $\overline{\delta}(S)$ and  $\underline{\delta}(S)$, as follows:
\[\overline{\delta}(S)={\limsup \limits_{x\rightarrow\infty}} ~ \frac{S(x)}{x},~~~~~~~~~\underline{\delta}(S)={\liminf \limits_{x\rightarrow\infty}} ~ \frac{S(x)}{x}.\]
We say that $S$ has density ${\delta}(S)$, when $\overline{\delta}(S)$= $\underline{\delta}(S)$= ${\delta}(S).$ The parameter of interest is the maximal density of an $M$-set, defined by

\begin{equation}
\nonumber \mu(M) := \sup \overline{\delta}(S),
\end{equation}
where the supremum is taken over all $M$-sets $S.$
Motzkin \cite{mot} posed the problem of finding the quantity $\mu(M)$. In 1973, Cantor and Gordon \cite{cantor1973sequences} proved that there exists a set $S$ such that $\delta(S)=\mu(M)$, when $M$ is finite. The following two lemmas proved in \cite{cantor1973sequences} and \cite{haralambis1977sets}, respectively, are useful for bounding $\mu(M)$.

\begin{lem}\label{a}
Let $ M=\{m_{1},m_{2},m_{3},\ldots\}$ and $c$ and $m$ be positive integers such that $\gcd(c,m)=1.$ Then
\begin{equation}
 \nonumber \mu(M)\geq\kappa(M):=\sup\limits_{(c,m)=1}(1/m)\min\limits_{k \geq 1}|c m_{k}|_{m},
 \end{equation}
where for an integer $x$ and a positive integer $m$, $|x|_{m} = |r|$ if  $x \equiv r \pmod m$ with $0 \leq |r| \leq m/2$.
\end{lem}

\begin{lem}\label{h}
Let $\alpha$ be a real number, $\alpha \in [0,1].$ If for any $M$-set $S$ with $0\in S$ there exists a positive integer $k$ such that $ S(k)\leq(k+1)\alpha,$ then $\mu(M)\leq \alpha.$
\end{lem}
For a finite set $M,$ by a remark of Haralambis \cite[Remark 1]{haralambis1977sets}, we can write $\kappa(M)$ as,
\begin{equation}\label{b}
\kappa(M)=\max_{\stackrel{m=m_{i}+m_{j}}{ 1 \leq k \leq \frac{m}{2}}}(1/m)\min\limits_{i}|k m_{i}|_{m},
\end{equation}
where $m_i,m_j$ are distinct elements of $M$.

The parameter $\kappa(M)$ which serves as a lower bound for $\mu(M)$, is related to the ``lonely runner conjecture". The  lonely runner conjecture is a long standing open conjecture on the Diophantine approximations, which was first posed by Wills \cite{Will67} and then independently by Cusick \cite{Cusick}. For the current developments and all related references on the conjecture, one can see the recent paper by Tao \cite{tao}.

As an another application of the Motzkin's maximal density problem, one can see that this problem is closely related to several coloring parameters of the distance graph generated by $M$. The study of Motzkin's density problem is equivalent to the study of the fractional chromatic number of distance graphs. A {\it fractional coloring} of a graph $G$ is a mapping $c$ which assigns to each independent set $I$ of $G$ (an ``independent set" of a graph is a set of pairwise nonadjacent vertices) a non-negative weight $c(I)$ such that for each vertex $x,~~\sum_{x \in I}c(I) \geq 1$. The \emph{fractional chromatic number} of $G$, denoted by $\chi_{f}(G)$, is the least total weight of a fractional coloring of $G$.

Let $M$ be a set of positive integers. The {\it distance graph} generated by $M$, denoted by $G(\mathbb{Z},M)$, has the set $\mathbb{Z}$ of all integers as the vertex set, and two vertices $x$ and $y$ are adjacent whenever $|x-y|\in M$. It was proved by Chang et al. \cite{Chang1999} that for any finite set $M$, the fractional chromatic number of the distance graph generated by $M$ is the reciprocal of the maximal density of $M$-sets. Precisely, they proved the next theorem.

\begin{thm}\label{t1}
 For any finite set $M$ of positive integers, $\mu(M)=1/\Large{\chi_{f}}(G(\mathbb{Z},M))$.
\end{thm}

The fractional chromatic number of a graph is related to another coloring parameter called the \emph{circular chromatic number} defined as follows: Let $k \geq 2d$ be positive integers. A $(k, d)$-coloring of a graph $G$ is a mapping, $c : V(G) \rightarrow \{0,1, \ldots, k-1\}$, such that $d \leq |c(u) - c(v)| \leq k-d$ for any $uv \in E(G)$. The circular chromatic number of $G$, denoted by $\chi_{c}(G)$, is the minimum ratio $k/d$ such that $G$ admits a $(k, d)$-coloring. Zhu \cite{Zhu96} proved that for any graph $G$,
\[\textit{$\chi_{f}$}(G)\leq \textit{$\chi_{c}$}(G)\leq \textit{$\chi$}(G)=\lceil \textit{$\chi_{c}$}(G)\rceil,\]
where $\chi(G)$ is the chromatic number of $G$ (the ``chromatic number" of a graph is the minimum number of colors required to color the vertices of the graph so that the adjacent vertices are assigned different colors).
Moreover, for a distance graph $G(\mathbb{Z},M)$, the following theorem \cite{Zhu96} connects the circular chromatic number of $G(\mathbb{Z},M)$ with $\kappa(M)$.

\begin{thm}\label{t2}
 For any finite set $M$ of positive integers, $\textit{$\chi_{c}$}(G(\mathbb{Z},M))\leq\frac{1}{\kappa(M)}.$
\end{thm}

Notice that $\kappa(M)$ gives a lower bound for $\mu(M)$ and the reciprocal of $\kappa(M)$ gives an upper bound for $\chi_{c}(G(\mathbb{Z},M))$.

The values and bounds of $\mu(M)$ for several special families of sets $M$ (\cite{cantor1973sequences, Chang1999, liu14, Griggs, Gupta, gupta1999density, haralambis1977sets, SurveyL, LiuRobinson, Aileen, liu2004fractional, Liu2008, pr21, pr20, ps22, pandey2011density, Pandey2011note, pandey2015some, sppintegers, sppudt}) have been studied. But, in general, only for $|M|\leq2$, complete solutions were given by Cantor and Gordon \cite{cantor1973sequences}.

Liu and Zhu \cite{liu2004fractional} studied $\mu(M)$ and $\kappa(M)$ for the families $M = \{x, y, x+y\}$ and $M = \{x, y,y-x, x+y\}$ with $x < y$. Motivated by these two families, we extend both of these families with one more element in each in a particular situation. More precisely, we study the values and bound for $\mu(M)$ and $\kappa(M)$ for the families $M=\{a,a+1,2a+1,n\}$ and $M=\{a,a+1,2a+1,3a+1,n\}$, where $a$ is a positive integer and $n$ is a sufficiently large positive integer.

We let $\mathbb{N}$ denote the set of positive integers and for $x \in \mathbb{R}$, $\|x\|$ is the distance of $x$ from the nearest integer, i.e., $\|x\| = \min \{x - \lfloor x \rfloor, \lceil x \rceil -x\}$. Using definition (\ref{b}) of $\kappa(M)$, we give lower bounds for $\kappa(M)$ for most of the sets in the families $M=\{a,a+1,2a+1,n\}$ and $M=\{a,a+1,2a+1,3a+1,n\}$ in Sections \ref{sec2} and \ref{sec3}, respectively. Upper bounds for $\kappa(m)$ are mentioned in the concluding remarks in Section \ref{sec4}.

Since we have $\mu(M) = \mu(tM)$ for any positive integer $t$, it is sufficient to consider the case $\gcd(M)=1$, which is already satisfied in both the families.

\section{The family $M=\{a,a+1,2a+1,n\}$}\label{sec2}

In Theorem \ref{car4}, we find lower bounds for $\kappa(M)$ and $\mu(M)$, for all but finitely many values of $n$. For a certain class of $n$-values, we also find the exact formulas for both of $\kappa(M)$ and $\mu(M)$ establishing equality between them.

\begin{obs}
Let $a$ be a  positive integer and $i$ be a nonnegative integer. Set
\begin{eqnarray*}
N_{1}^{(i)} &=& \big\{(3a+2)a +(3a+1)i+l : 0\leq l\leq a \big\},\\
N_{2}^{(i)} &=& \big\{(3a+2)a +(3a+1)i+a+1+l : 0\leq l\leq a+\big{\lfloor} \frac{a}{3}\big{\rfloor} \big\},\\
N_{3}^{(i)} &=& \big\{(3a+2)a +(3a+1)i+2(a+1)+\big{\lfloor} \frac{a}{3}\big{\rfloor}+l : 0 \leq l\leq a-2-\big{\lfloor} \frac{a}{3}\big{\rfloor} \big\}.
\end{eqnarray*}
Then $N_{1}^{(i)}, N_{2}^{(i)}, N_{3}^{(i)}$ are pairwise disjoint sets and  $\bigcup_{i \geq 0} N_{1}^{(i)}\cup N_{2}^{(i)}\cup N_{3}^{(i)}= \mathbb{N} \setminus \{(3a+2)a-1\}$.
\end{obs}

\begin{thm}\label{car4}
 Let $M=\{a,a+1,2a+1,n\}$. Then
\begin{eqnarray*}
\kappa(M)
\begin{cases}
 = \frac{a}{3a+1} = \mu(M), & \text{if} \quad n\in N_{1}^{(i)} \bigcup_{j=0}^{a-1} \{a+(3a+1)j+l : 0 \leq l \leq a\};\\
 \\
 \geq \frac{a(a+1+i)}{a+n}, & \text{if} \quad n\in N_{2}^{(i)};\\
 \\
 \geq \frac{a(a+1+i)+\lfloor \frac{a}{3}\rfloor+1+l}{2a+n+1}, & \text{if} \quad n\in N_{3}^{(i)}.
\end{cases}
\end{eqnarray*}
\end{thm}

\begin{proof}
\textbf{Case (i):} Let $n\in N_{1}^{(i)} \bigcup_{j=0}^{a-1} \{a+(3a+1)j+l : 0 \leq l \leq a\}$. Let $m=3a+1$ and $x=1$. Then, we have
\begin{eqnarray*}
 ax   &\equiv &  a{\pmod m},\\
(2a+1)x &\equiv & -ax\equiv -a{\pmod m},\\
 (a+1)x & = & (2a+1-a)x \equiv (a+1){\pmod m},\\
\text{and} \quad nx & = & n \equiv a+l{\pmod m}.
\end{eqnarray*}
Since $a \leq a+l \leq 2a < m-a$, we have
 $$\min \{|ax|_{m},~|(a+1)x|_{m},~|(2a+1)x|_{m},~ |nx|_{m}\}= a.$$
So, we get
  $$\kappa(M)\geq \frac{m-1}{3m} = \frac{a}{3a+1}.$$
  Let $B= \{0,a,a+1,2a+1\}$ and $A_t = \{t, a+1+t, 2a+1+t\}$ for $1 \leq t \leq a-1$. Then the sets $B$ and $A_t$ partition the set $\{0,1,2, \ldots, 3a\}$. Furthermore, if $S$ is any $M$-set with $0 \in S$, then $|S \cap B|=1$ and $|S \cap A_t| \leq 1$ for $1 \leq t \leq a-1$. So, $|S \cap \{0,1,2, \ldots, 3a\}| \leq 1+a-1 = a$. That is $S(3a) \leq a$, hence by the Lemma \ref{c} $\mu(M) \leq \frac{a}{3a+1}$. Thus, we have $\frac{a}{3a+1} \leq \kappa(M) \leq \mu(M) \leq \frac{a}{3a+1}$. Hence
  \[\kappa(M) = \mu(M) = \frac{a}{3a+1}.\]

\textbf{Case (ii):} Let $n\in N_{2}^{(i)}$. Let $m=a+n$ and $x=a+1+i$. Then, we have
\begin{eqnarray*}
ax &\equiv & a(a+1+i){\pmod  m},\\
nx &\equiv & -ax\equiv -a(a+1+i){\pmod m},\\
(a+1)x &\equiv &(a+1)(a+1+i){\pmod m},\\
(2a+1)x &\equiv &(2a+1)(a+1+i){\pmod m}.
\end{eqnarray*}
Since $a(a+1+i)\leq (2a+1)(a+1+i)\leq m-a(a+1+i)$, we have
$$\min \{|ax|_{m},|(a+1)x|_{m}, |(2a+1)x|_{m}, |nx|_{m} \} = a(a+1+i).$$
So, we get
$$\kappa(M)\geq \frac{a(a+1+i)}{m} =  \frac{a(a+1+i)}{a+n}.$$

\textbf{Case (iii):} Let $n\in N_{3}^{(i)}$. Let $m=2a+1+n$ and $x=a+2+i$. Then, we have
\begin{eqnarray*}
ax &\equiv & a(a+2+i){\pmod  m},\\
(a+1)x &\equiv & (a+1)(a+2+i){\pmod m},\\
(2a+1)x &= &(2a+1)(a+2+i)\\
 &\equiv & (2a+1)(a+2+i)-m{\pmod m}\\
 &= &-\big{(}a(a+1)+ai+\big{\lfloor} \frac{a}{3}\big{\rfloor}+1+l\big{)}{\pmod m},\\
nx &\equiv &-(2a+1)x\equiv a(a+1+i)+\big{\lfloor} \frac{a}{3}\big{\rfloor}+1+l{\pmod m}.
\end{eqnarray*}
Since
\begin{eqnarray*}
 \min \{|ax|_{m},|(a+1)x|_{m}, |(2a+1)x|_{m}, |nx|_{m} \} = a(a+1+i)+\lfloor \frac{a}{3}\rfloor+1+l.
 \end{eqnarray*}
So, we get
  $$\kappa(M)\geq \frac{ a(a+1+i)+\lfloor \frac{a}{3}\rfloor+1+l}{m} =  \frac{ a(a+1+i)+\lfloor \frac{a}{3}\rfloor+1+l}{2a+n+1}.$$
This completes the proof.
\end{proof}

\section{The family $M=\{a,a+1,2a+1,3a+1,n\}$}\label{sec3}

In this section, we find lower bounds for $\kappa(m)$ depending on $a \equiv 0, 1, 2,$ or $3  \pmod 4$, respectively, in Theorems \ref{sec31}, \ref{$a=4k+1$}, \ref{$a=4k+2$}, and \ref{$a=4k+3$}, for sufficiently large values of $n$. 
\begin{obs}
Let $a$ be a positive integer with $a \equiv 0 \pmod 4$ and $i$ be a nonnegative integer. Set
\begin{eqnarray*}
O_{1}^{(i)} &=& \big\{(8a+3)a +(4a+1)i+l : 0\leq l\leq 2a\big\},\\
O_{2}^{(i)} &=& \big\{(8a+3)a +(4a+1)i+2a+1+l : 0\leq l\leq 2a-\frac{a}{2} \big\},\\
O_{3}^{(i)} &=& \big\{(8a+3)a +(4a+1)i+2(2a+1)-\frac{a}{2}+l : 0 \leq l\leq \frac{a}{2} -2\big\}.
\end{eqnarray*}
Then $O_{1}^{(i)}, O_{2}^{(i)}, O_{3}^{(i)}$ are pairwise disjoint sets and  $\bigcup_{i \geq 0} O_{1}^{(i)}\cup O_{2}^{(i)}\cup O_{3}^{(i)}= \mathbb{N} \setminus \{a(8a+3)-1\}$.
\end{obs}

\begin{thm}\label{sec31}
 Let $M=\{a,a+1,2a+1,3a+1,n\}$, with $a=4k$, $k\geq 1$, then for $n\geq (8a+3)a$,
\begin{eqnarray*}
\kappa(M)\geqslant
\begin{cases}
 \frac{a}{4a+1}, & \text{if} \quad n\in O_{1}^{(i)};\\
 \\
 \frac{a(2a+1+i)}{a+n}, & \text{if} \quad n\in O_{2}^{(i)};\\
 \\
 \frac{a(2a+1+i)+\frac{a}{2}+1+l}{3a+1+n}, & \text{if} \quad n\in O_{3}^{(i)}.
\end{cases}
\end{eqnarray*}
\end{thm}

\begin{proof}
\textbf{Case (i):} Let $n\in O_{1}^{(i)}$. Let $m=4a+1$ and $x=1$. Then, we have
\begin{eqnarray*}
ax  &\equiv & a{\pmod m},\\
(3a+1)x &\equiv &-ax\equiv -a{\pmod m},\\
 (a+1)x & \equiv & (a+1){\pmod m},\\
 (2a+1)x & \equiv & -2a{\pmod m},\\
\text{and} \qquad nx & = & n \equiv a+l{\pmod m}.
\end{eqnarray*}
Since $a+l\in [a,m-a],$ we have
$$\min \{|ax|_{m},~|(a+1)x|_{m},~|(2a+1)x|_{m},~|(3a+1)x|_{m},~ |nx|_{m}\}= a.$$
 So,
  $$\kappa(M)\geq \frac{a}{m} = \frac{a}{4a+1}.$$

\textbf{Case (ii):} Let $n\in O_{2}^{(i)}$. Let $m=a+n$ and $x=2a+1+i$. Then, we have
\begin{eqnarray*}
ax &\equiv & a(2a+1+i){\pmod  m},\\
nx &\equiv & -ax\equiv -a(2a+1+i){\pmod m},\\
(a+1)x &\equiv & (a+1)(2a+1+i){\pmod m},\\
(2a+1)x &\equiv & (2a+1)(2a+1+i){\pmod m},\\
(3a+1)x &\equiv & (3a+1)(2a+1+i){\pmod m}.
\end{eqnarray*}
Since $a(2a+1+i)\leq (3a+1)(a+1+i)\leq m - a(2a+1+i)$, we have
\begin{eqnarray*}
 \min \{|ax|_{m},|(a+1)x|_{m}, |(2a+1)x|_{m},|(3a+1)x|_{m}, |nx|_{m} \} = a(2a+1+i).
 \end{eqnarray*}
So,
  $$\kappa(M)\geq \frac{a(2a+1+i)}{m} = \frac{a(2a+1+i)}{a+n}.$$

 \textbf{Case (iii):} Let $n\in O_{3}^{(i)}$. Let $m=3a+1+n$ and $x=2a+2+i$. Then, we have
 \begin{eqnarray*}
ax &\equiv & a(2a+2+i){\pmod  m},\\
(a+1)x &\equiv & (a+1)(2a+2+i){\pmod m},\\
(2a+1)x &\equiv & (2a+1)(2a+2+i){\pmod m}\\
 &\equiv & -\big{(}a(4a+3)+2ai+\frac{a}{2}+1+l\big{)}{\pmod m},\\
(3a+1)x &\equiv &(3a+1)(2a+2+i){\pmod m}\\
 &\equiv &-\big{(}a(2a+1)+ai+\frac{a}{2}+1+l\big{)}{\pmod m},\\
nx &\equiv & -(3a+1)x{\pmod m}.
 \end{eqnarray*}
Since $a(4a+3)+2ai+\frac{a}{2}+1+l \leq \frac{m}{2}$ if and only if $l\leq \frac{5a}{2}+1+i,$ which is always true.
 So,
\begin{eqnarray*}
 \min \{|ax|_{m},|(a+1)x|_{m}, |(2a+1)x|_{m},|(2a+1)x|_{m}, |nx|_{m} \} = a(2a+1)+ai+\frac{a}{2}+1+l.
 \end{eqnarray*}
 Hence
  $$\kappa(M)\geq \frac{ a(2a+1)+ai+\frac{a}{2}+1+l}{m} =  \frac{a(2a+1+i)+\frac{a}{2}+1+l}{3a+1+n}.$$
This completes the proof.
\end{proof}

\begin{obs}
Let $a$ be a positive integer with $a \equiv 1 \pmod 4$ and $i$ be a nonnegative integer. Set
\begin{eqnarray*}
P_{1}^{(i)} &=& \big\{(16a+17)a+3 +(4a+1)i+l : 0\leq l\leq 2a\big\},\\
P_{2}^{(i)} &=& \bigg\{(16a+17)a+3 +(4a+1)i+2a+1+l : 0\leq l\leq \frac{3a-1}{2} \bigg\},\\
P_{3}^{(i)} &=& \bigg\{(16a+17)a+3 +(4a+1)i+\frac{7a+3}{2}+l : 0 \leq l\leq \frac{a-3}{2} \bigg\}.
\end{eqnarray*}
Then $P_{1}^{(i)}, P_{2}^{(i)}, P_{3}^{(i)}$ are pairwise disjoint sets and  $\bigcup_{i \geq 0} P_{1}^{(i)}\cup P_{2}^{(i)}\cup P_{3}^{(i)}= \mathbb{N} \setminus \{a(16a+17)+2\}$.
\end{obs}

\begin{thm}\label{$a=4k+1$}
 Let $M=\{a,a+1,2a+1,3a+1,n\}$ with $a=4k+1$, $k \geq 0$. Then for $n\geq (16a+17)a+3$,
\begin{eqnarray*}
\kappa(M)\geqslant
\begin{cases}
 \frac{a}{4a+1}, & \text{if} ~n\in P_{1}^{(i)};\\
 \\
 \frac{4a(a+1)+ai}{a+n}, & \text{if}~n\in P_{2}^{(i)};\\
 \\
 \frac{4a(a+1)+\frac{a+1}{2}+ai+l}{3a+n+1}, & \text{if}~n\in P_{3}^{(i)}.
\end{cases}
\end{eqnarray*}
\end{thm}

\begin{proof}
\textbf{Case (i):} Let $n\in P_{1}^{(i)}$. Let $m=4a+1$ and $x=1$. Then, we have
\begin{eqnarray*}
ax  &\equiv & a{\pmod m},\\
(3a+1)x &\equiv & -ax\equiv -a{\pmod m},\\
 (a+1)x & \equiv & (a+1){\pmod m},\\
 (2a+1)x & \equiv & -2a{\pmod m},\\
\text{and} \qquad nx &\equiv&  a+l{\pmod m}.
\end{eqnarray*}
Since $a\leq a+l< m-a,$ we have
$$\min \{|ax|_{m},~|(a+1)x|_{m},~|(2a+1)x|_{m},~|(3a+1)x|_{m},~ |nx|_{m}\}= a.$$
 So,
  $$\kappa(M)\geq \frac{a}{m}=\frac{a}{4a+1}.$$

\textbf{Case (ii):} Let $n\in P_{2}^{(i)}$. Let $m=a+n$ and $x=4(a+1)+i$. Then, we have
\begin{eqnarray*}
ax &\equiv & a(4a+4+i){\pmod  m},\\
nx &\equiv & -ax\equiv -a(4a+4+i){\pmod m},\\
(a+1)x &\equiv & (a+1)(4a+4+i){\pmod m},\\
(2a+1)x &\equiv & (2a+1)(4a+4+i){\pmod m},\\
(3a+1)x &\equiv & (3a+1)(4a+4+i){\pmod m}.
\end{eqnarray*}
Since $~~a(4a+4+i)< (3a+1)(4a+4+i)\leq m-a(4a+4+i),$ we have
\begin{eqnarray*}
 \min \{|ax|_{m},|(a+1)x|_{m}, |(2a+1)x|_{m},|(3a+1)x|_{m}, |nx|_{m} \} = a(4a+4+i).
 \end{eqnarray*}
 So,
  $$\kappa(M)\geq \frac{a(4a+4+i)}{m}=\frac{4a(a+1)+ai}{a+n}.$$
 \textbf{Case (iii):} Let $n\in P_{3}^{(i)}$. Let $m=3a+1+n$ and $x=4a+5+i$. Then, we have
 \begin{eqnarray*}
ax &\equiv & a(4a+5+i){\pmod  m},\\
(a+1)x &\equiv & (a+1)(4a+5+i){\pmod m},\\
(2a+1)x&\equiv & (2a+1)(4a+5+i){\pmod m}\\
& \equiv &(2a+1)(4a+5+i)-m{\pmod m}\\
& \equiv & -\bigg{(}2a(4a+5)+2ai-\frac{a-1}{2}+l\bigg{)}{\pmod m},\\
(3a+1)x &\equiv & (3a+1)(4a+5+i){\pmod m}\\
& \equiv& (3a+1)(4a+5+i)-m{\pmod m}\\
& \equiv & -\bigg{(}4a(a+1)+ai+\frac{a+1}{2}+l\bigg{)}{\pmod m},\\
nx &\equiv & -(3a+1)x\equiv 4a(a+1)+ai+\frac{a+1}{2}+l{\pmod m}.
\end{eqnarray*}
Since
 $2a(4a+5)+2ai-\frac{a-1}{2}+l \leq \frac{m}{2}$ if and only if $l\leq \frac{9a-1}{2}+5+i,$ which is true.
 So, we have
\begin{eqnarray*}
 \min \{|ax|_{m},|(a+1)x|_{m}, |(2a+1)x|_{m},|(2a+1)x|_{m}, |nx|_{m} \} = 4a(a+1)+ai+\frac{a+1}{2}+l.
 \end{eqnarray*}
 Hence
  $$\kappa(M)\geq \frac{ 4a(a+1)+ai+\frac{a+1}{2}+l}{m}=\frac{4a(a+1)+\frac{a+1}{2}+ai+l}{3a+n+1}.$$

This completes the proof.
\end{proof}

\begin{thm}\label{$a=4k+2$}
 Let $M=\{a,a+1,2a+1,3a+1,n\}$ with $a=4k+2$, $k \geq 0$. Then
\begin{eqnarray*}
\kappa(M)\geqslant
\begin{cases}
 \frac{1}{4}, & \text{if} ~n \not\equiv 0 \pmod 4;\\
 \\
 \frac{m-(2a+1)}{4(2a+n+1)}, & \text{if}~n \equiv 0 \pmod 4.
 \end{cases}
\end{eqnarray*}
\end{thm}

\begin{proof}
\textbf{Case (i):} Let $n \not\equiv 0 \pmod 4$. Let $m=5a+2$ and $x=\frac{m}{4}$. Then, we have
\begin{eqnarray*}
 ax   &= &  (4k+2)\frac{m}{4} \equiv\frac{m}{2}{\pmod m},\\
(a+1)x &= & ax+x\equiv \frac{3m}{4} \equiv -\frac{m}{4}{\pmod m},\\
 (2a+1)x &= & 2ax+x \equiv \frac{am}{2}+\frac{m}{4} \equiv \frac{m}{4}{\pmod m},\\
 (3a+1)x &\equiv & -(2a+1)x\equiv -\frac{m}{4}{\pmod m},\\
\text{and} \qquad nx &\equiv & \frac{nm}{4} \pmod m.
\end{eqnarray*}
Since $n\not\equiv 0 \pmod 4$,  $nx \equiv \frac{m}{4} $ or $\frac{m}{2}$ or $-\frac{m}{4} {\pmod m}$.
 We have,
 $$\min \{|ax|_{m},~|(a+1)x|_{m},~|(2a+1)x|_{m},~|(3a+1)x|_{m},~ |nx|_{m}\}= \frac{m}{4}.$$
 Hence
  $$\kappa(M)\geq \frac{1}{4}.$$

\textbf{Case (ii):} Let $n \equiv 0 \pmod 4$. Let $m=2a+n+1$ and $x=\frac{m-1}{4}$. Then, we have
\begin{eqnarray*}
 ax   &= &  (4k+2)\big{(}\frac{m-1}{4}\big{)} \equiv -k+\frac{m-1}{2} \equiv \frac{2m-a}{4}{\pmod m},\\
(a+1)x &= & ax+x\equiv \frac{3m-a-1}{4}=-\frac{m+a+1}{4}{\pmod m},\\
 (2a+1)x &= & 2ax+x \equiv \frac{5m-2a-1}{4}=\frac{m-(2a+1)}{4}{\pmod m},\\
 (3a+1)x &= & \frac{3m-(3a+1)}{4} \equiv -\frac{m+3a+1}{4}{\pmod m},\\
\text{and} \qquad nx &\equiv & -(2a+1)x \equiv   -\frac{m-2a-1}{4}{\pmod m}.\\
\end{eqnarray*}
 So, we have,
 $$\min \{|ax|_{m},~|(a+1)x|_{m},~|(2a+1)x|_{m},~|(3a+1)x|_{m},~ |nx|_{m}\}= \frac{m-(2a+1)}{4}.$$
 Hence
  $$\kappa(M)\geq \frac{m-(2a+1)}{4m}=\frac{m-(2a+1)}{4(2a+n+1)}.$$
  This completes the proof.
\end{proof}

\begin{thm}\label{$a=4k+3$}
 Let $M=\{a,a+1,2a+1,3a+1,n\}$ with $a=4k+3$, $k\geq 0$. Let $L_i = \bigcup\limits_{t=0}^{4}[\frac{(4t+i+1)(5a+2)-(5i+1)}{20}, \frac{(4t+3+i)(5a+2)+1-5i}{20}]$ and
 $S_{i}= \{4\big{(}q(5a+2)+r\big{)}+i : q\in \mathbb{Z}, r\in [0, 5a+1] \cap L_i\}$. If $n\in \bigcup\limits_{i=0}^{3}S_{i},$ then
\begin{eqnarray*}
\kappa(M)\geqslant \frac{5a+1}{4(5a+2)}.
\end{eqnarray*}
\end{thm}

\begin{proof}
Let $a=4k+3,~~ k(\geq 0)\in \mathbb{Z}$ and $m=5a+2.$ Let $x = \frac{m-5}{4}.$ Then
\begin{eqnarray*}
ax &=& a\big{(}\frac{m-5}{4}\big{)}\equiv -5k+3\big{(}\frac{m-5}{4}\big{)}\equiv -\big{(}\frac{m-1}{2}\big{)}{\pmod m}, \\
(a+1)x &=& -5k+4 \big{(}\frac{m-5}{4}\big{)} \equiv -5(k+1)\equiv - \big{(}\frac{m+3}{4}\big{)}{\pmod m},\\
(2a+1)x &=& (8k+7)\big{(}\frac{m-5}{4}\big{)}\equiv -10k+3 \big{(}\frac{m-5}{4}\big{)}-5\equiv \frac{m-1}{4}{\pmod m},\\
(3a+1)x &\equiv & -(2a+1)x \equiv - \frac{m-1}{4}{\pmod m},\\
\text{Also}~~ nx &=& \big{(}4(qm+r)+i\big{)}\big{(}\frac{m-5}{4}\big{)} \equiv -5r+i\big{(}\frac{m-5}{4}\big{)}{\pmod m}.
\end{eqnarray*}
Now~~~ $$tm+\frac{m-1}{4}\leq 5r -i\big{(}\frac{m-5}{4}\big{)} \leq (t+1)m- \frac{m-1}{4} $$
if and only if $$r \in \bigg{[}\frac{(4t+i+1)m-(5i+1)}{20}, \frac{(4t+3+i)m+1-5i}{20}\bigg{]}.$$
Thus if $n\in \bigcup\limits_{i=0}^{3}S_{i},$ then
$$\min \{|ax|_{m},~|(a+1)x|_{m},~|(2a+1)x|_{m},~|(3a+1)x|_{m},~ |nx|_{m}\}= \frac{m-1}{4}.$$
Hence $$\kappa(M)\geqslant \frac{m-1}{4m}=\frac{5a+1}{4(5a+2)}.$$
  This completes the proof.
\end{proof}

\section{Concluding Remarks}\label{sec4}
\begin{enumerate}
  \item Let $X_j = \{a+(4a+1)j+l : 0 \leq l \leq 2a\}$. Then using the same proof as in Case (i) of
  Theorem \ref{sec31}, we have that if $M=\{a,a+1,2a+1,3a+1,n\}$, with $a=4k$, $k\geq 1$, and $n \in \bigcup_{j=0}^{2a} X_{j}$, then $\kappa(M) \geqslant \frac{a}{4a+1}$.
  \item Let $X_j = \{a+(4a+1)j+l : 0 \leq l \leq 2a\}$. Then using the same proof as in Case (i) of
  Theorem \ref{$a=4k+1$}, we have that if $M=\{a,a+1,2a+1,3a+1,n\}$, with $a=4k+1$, $k\geq 0$, and $n \in \bigcup_{j=0}^{4a+3} X_{j}$, then $\kappa(M) \geqslant \frac{a}{4a+1}$.
  \item Since $\kappa(\{a,a+1,2a+1,3a+1,n\}) \leq \kappa(\{a,a+1,2a+1,3a+1\}) \leq \mu(\{a,a+1,2a+1,3a+1\})$, we have that if $a$ is even, then by a theorem of Liu and Zhu \cite[Theorem 4.1]{liu2004fractional}, $\kappa(\{a,a+1,2a+1,3a+1,n\}) \leq \frac{1}{4}$. Furthermore, if $a$ is odd, then by a lemma of Liu and Zhu \cite[Lemma 5.1]{liu2004fractional}, we have $\kappa(\{a,a+1,2a+1,3a+1,n\}) < \frac{1}{4}$.
\end{enumerate}

{ \bf {Acknowledgements.}}~ The first author is thankful to the Council of Scientific and Industrial Research (CSIR) for providing the research grant no. 25(0314)/20/EMR-II.


\begin{thebibliography}{999}



\bibitem{Serra}
J. Barajas and O. Serra, {\it Distance graphs with maximum chromatic number}, Disc. Math., {\bf 308} (2008), 1355--1365.

\bibitem{cantor1973sequences} D. G. Cantor and B. Gordon,
{\it  Sequences of integers with missing differences},
J. Combin. Theory Ser. A, {\bf 14} (1973), 281--287.

%


\bibitem{Chang1999}
G. Chang, D. D.-F. Liu, and X. Zhu, {\it Distance graphs and $T$-colorings}, J. Combin. Theory Ser. B, {\bf 75} (1999), 159--169.


%



\bibitem{liu14}
D. Collister and D. D.-F. Liu, Study of $\kappa (D)$ for $D=\{2, 3, x, y\}$, {\it  Combinatorial Algorithms. Lecture Notes in Computer Science, Proceeding of the 25th International Workshop}, Springer (2014), 250--261.

\bibitem{Cusick}
T. W. Cusick, {\it View-obstruction problems in $n$-dimensional geometry}, J. Combin. Theory Ser. A, {\bf 16} (1974), 1--11.




\bibitem {Griggs}
J. R. Griggs and D. D.-F. Liu,
\textit{The channel assignment problem for mutually adjacent sites},
J. Combin. Theory Ser. A,
 {\bf 68} (1994), 169--183.

\bibitem
{Gupta}
S. Gupta,
{\em Sets of integers with missing differences},
J. Combin. Theory Ser. A, {\bf 89} (2000), 55--69.

\bibitem {gupta1999density}
S. Gupta and A. Tripathi,
\textit{Density of $M$-sets in arithmetic progression},
Acta Arith.,
{\bf 89} (1999), 255--257.

\bibitem
{haralambis1977sets}
N.~M.~Haralambis,
{\em Sets of integers with missing differences},
J. Combin. Theory Ser. A, {\bf 23} (1977), 22--23.




%


\bibitem
{SurveyL}
D.~Liu,
{\em From rainbow to the lonely runner: A survey on coloring parameters of distance graphs},
Taiwanese J. Math., {\bf 12} (2008), 851--871.

\bibitem{LiuRobinson}
D. D.-F. Liu and G. Robinson, Sequences of integers with three missing
separations, {\it European J. Combin.}, {\bf 85} (2020), 103056.

\bibitem
{Aileen}
D. Liu and A. Sutedja,
{\it Chromatic number of distance graphs generated by the sets $\{2, 3, x, y\}$},
J. Comb. Optim., {\bf 25} (2013), 680--693.


\bibitem{liu2004fractional}
D. D.-F. Liu and X. Zhu, {\it Fractional chromatic number for distance graphs with large clique size}, J. Graph Theory, {\bf 47} (2004), 129--146.


\bibitem {Liu2008}
D. D.-F. Liu and X. Zhu,
\textit{Fractional chromatic number of distance graphs generated by two-interval sets},
European J. Combin.,
{\bf 29} (2008), no. 7, 1733--1742.

\bibitem {mot}
T. S. Motzkin,
{\it Problems collection} (unpublished).

\bibitem{pr21}
R. K. Pandey and N. Rai,
\textit{Density of sets with missing differences and applications},
Math. Slovaca, \textbf{71} (2021), 595--614.

\bibitem{pr20} R. K. Pandey and N. Rai, \textit{Maximal density of sets with missing differences and various coloring parameters of distance graphs}, Taiwanese J. Math., \textbf{24} (2020), 1383--1397.

\bibitem{ps22} R. K. Pandey and A. Srivastava, \textit{Maximal density of integral sets with missing differences and the kappa values}, Taiwanese J. Math., \textbf{26} (2022), 17--32.



\bibitem {pandey2015some}
R.~K.~Pandey and A.~Tripathi,
\textit{A note on the density of $M$-sets in geometric sequence},
Ars Comb.,
{\bf CXIX} (2015), 221-224.

\bibitem {Pandey2011note}
R.~K.~Pandey and A.~Tripathi,
\textit{A note on a problem of Motzkin regarding density of integral sets with missing differences},
J. Integer Sequences,
{\bf 14} (2011), no.6, Article 11.6.3.

\bibitem{pandey2011density}
R.~K.~Pandey and A.~Tripathi,
{\it On the density of integral sets with missing differences from sets related to arithmetic progressions},
J. Number Theory, {\bf 131} (2011), 634–-647.

\bibitem {Rabinowitz}
J. H. Rabinowitz and V. K. Proulx, {\it An asymptotic approach to the channel assignment problem}, SIAM J. Alg. Disc. Methods, {\bf 6} (1985), 507--518.

\bibitem {sppintegers}
A. Srivastava, R. K. Pandey, and O. Prakash,  {\it On the maximal density of integral sets whose differences avoiding the weighted Fibonacci numbers}, Integers, {\bf 17} (2017), A48.

\bibitem {sppudt}
 A. Srivastava, R. K. Pandey, and O. Prakash,  {\it Motzkin's maximal density and related chromatic numbers}, Unif. Distrib. Theory, {\bf 13} (2018), no. 1, 27--45.

\bibitem{tao}
T. Tao, {\em Some remarks on the lonely runner conjecture},
Contrib. Discrete Math., {\bf 13(2)} (2018), 1--31.


%


%
%


\bibitem{Will67}
J. M. Wills, {\it Zwei S\"{a}tze \"{u}ber inhomogene diophantische Approximation
von Irrationalzahlen}, Monatsh. Math., {\bf 71} (1967), 263--269.

\bibitem{Zhu96}
X.~Zhu, {\it Circular chromatic number: A survey},  Disc. Math., {\bf 229} (2001), 371--410.

%




\end{thebibliography}
\end{document}